\documentclass{amsart}

\usepackage[OT2,T1]{fontenc}
\usepackage{tilmansdef}
\usepackage{tikz}
\usepackage[colorlinks=true,linkcolor=blue,citecolor=blue]{hyperref} 
\usepackage{tensor}
\usetikzlibrary{matrix,arrows,decorations,cd}

\usepackage{mathtools,stmaryrd}

\DeclareFontFamily{OT1}{pzc}{}
\DeclareFontShape{OT1}{pzc}{m}{it}{<-> s * [1.10] pzcmi7t}{}
\DeclareMathAlphabet{\mathpzc}{OT1}{pzc}{m}{it}

\DeclareMathOperator{\Pol}{Pol}
\DeclareMathOperator{\pol}{pol}

\newcommand{\FSch}[1]{\operatorname{FSch}_{#1}}

\newcommand{\DMod}[1]{\operatorname{Dmod}_{#1}}

\newcommand{\DDModF}[1]{\mathbb{D}\operatorname{mod}_{#1}^{F}}

\newcommand{\DieuFor}{\mathbb D^{\operatorname{f}}}

\newcommand{\FGps}[1]{\operatorname{Fgps}_{#1}} 
\newcommand{\FGpsc}[1]{\operatorname{Fgps}^c_{#1}} 
\newcommand{\FGpse}[1]{\operatorname{Fgps}^{e}_{#1}} 
\newcommand{\AbSch}[1]{\operatorname{AbSch}_{#1}} 
\newcommand{\AbSchu}[1]{\operatorname{AbSch}^u_{#1}} 
\newcommand{\AbSchm}[1]{\operatorname{AbSch}^m_{#1}} 
\newcommand{\Hopf}[1]{\operatorname{Hopf}_{#1}}

\renewcommand{\ell}{\mathpzc{l}}

\DeclareMathOperator{\frob}{frob}

\DeclareMathOperator{\Gal}{Gal}

\DeclareMathOperator{\nil}{nil}

\newcommand{\modtensor}[2]{\rtimes}
\newcommand{\modcotensor}[2]{\hom}
\newcommand{\Spf}[1]{\operatorname{Spf}\left({#1}\right)}
\newcommand{\Reg}[1]{\mathcal{O}_{#1}}

\DeclareMathOperator{\alg}{alg}

\newtheorem{notation*}[equation]{Notation}

\newtheorem{innercustomthm}{Theorem}
\newenvironment{customthm}[1]%
  {\renewcommand\theinnercustomthm{\arabic{section}.\arabic{equation}#1}\innercustomthm}%
  {\endinnercustomthm}

\newtheorem{innercustomlemma}{Lemma}
\newenvironment{customlemma}[1]%
  {\renewcommand\theinnercustomlemma{\arabic{section}.\arabic{equation}#1}\innercustomlemma}%
  {\endinnercustomlemma}
\title{Affine and formal abelian group schemes on $p$-polar rings}
\author{Tilman Bauer}
\address{KTH Royal Institute of Technology\\Institutionen för matematik\\Lindstedtsvägen 25\\10044 Stockholm\\Sweden}
\email{tilmanb@kth.se}
\date\today
\dateposted\today

\begin{document}

\begin{abstract}
We show that the functor of $p$-typical co-Witt vectors on commutative algebras over a perfect field $k$ of characteristic $p$ is defined on, and in fact only depends on, a weaker structure than that of a $k$-algebra. We call this structure a $p$-polar $k$-algebra. By extension, the functors of points for any $p$-adic affine commutative group scheme and for any formal group are defined on, and only depend on, $p$-polar structures. In terms of abelian Hopf algebras, we show that a cofree cocommutative Hopf algebra can be defined on any $p$-polar $k$-algebra $P$, and it agrees with the cofree commutative Hopf algebra on a commutative $k$-algebra $A$ if $P$ is the $p$-polar algebra underlying $A$; a dual result holds for free commutative Hopf algebras on finite $k$-coalgebras.
\end{abstract}

\maketitle

\section{Introduction} \label{sec:intro}

Let $p$ be a prime. We consider the following generalizations of the notion of a $k$-algebra $A$:

\begin{defn}
Let $k$ be a (commutative) ring and $A$ a $k$-module. A \emph{$p$-polar $k$-algebra structure} on $A$ is a symmetric $k$-multilinear map $\mu\colon A^{\otimes_k p} \to A$ such that
\[
\mu(\mu(x_1,\dots,x_p),x_{p+1},\dots,x_{2p-1})\tag{ASSOC} \text{ is $\Sigma_{2p-1}$-invariant}
\]
for the permutation action of the symmetric group $\Sigma_{2p-1}$ on $x_1,\dots,x_{2p-1} \in A$. We will call a $p$-polar $\Z$-algebra a $p$-polar ring.
\end{defn}

A morphism of $p$-polar algebras is the evident structure-preserving map, making $p$-polar algebras into a category $\Pol_p(k)$. We denote by $\pol_p(k)$ the full subcategory of $p$-polar algebras which are of finite length as $k$-modules.

Clearly, any $k$-algebra $R$ gives rise to a $p$-polar structure for each $p$ by restriction, called its \emph{polarization} $\pol(R)$. If $R$ is of finite length then so is $\pol(R)$, so polarization gives functors
\[
\pol\colon \Alg_k \to \Pol_p(k) \quad \text{and} \quad \pol\colon \alg_k \to \pol_p(k) 
\]
from the categories of $k$-algebras and finite-length $k$-algebras, respectively.

Our main results concern the following categories over a perfect field $k$ of characteristic $p$:
\begin{itemize} 
	\item the category $\AbSch{k}$ of affine, commutative group schemes (``affine groups''), anti-equivalent to the category of bicommutative Hopf algebras over $k$;
	\item its full subcategory $\AbSch{k}^p$ of $p$-adic groups, i.~e. group schemes with values in abelian pro-$p$-groups). These correspond do bicommutative Hopf algebras $H$ that are \emph{$p$-adic}: $H \cong \colim_n H[p^n]$, where $H[p^n]$ denotes the kernel of the endomorphism $[p^n]$ of $H$ representing multiplication by $p^n$.
	\item the category $\FGps{k}$ of affine, commutative, formal group schemes (``formal groups''). These are ind-representable functors from the category $\alg_k$ of finite-dimensional $k$-algebras to abelian groups, as in \cite{fontaine:groupes-divisibles}. The category $\FGps{k}$ is anti-equivalent to the category of \emph{complete Hopf algebras}, which is the category of cogroup objects in the category of pro-finite dimensional $k$-algebras (with monoidal structure given by the profinitely completed tensor product);
	\item its full subcategory $\FGps{k}^p$ of formal $p$-group schemes, i.~e. formal groups taking values in $p$-groups.
\end{itemize}

The categories $\FGps{k}$ and $\AbSch{k}$ are anti-equivalent by Cartier duality, represented by taking (continuous) $k$-linear duals at the level of (complete) Hopf algebras, and this anti-equivalence restricts to an anti-equivalence between $\FGps{k}^p$ and $\AbSch{k}^p$.

\stepcounter{equation} 
\begin{customthm}{A}\label{thm:affineschemefactorization}
Let $k$ be a perfect field of characteristic $p$ and $G \in \AbSch{k}^p$. Then the functor of points of $G$ factors through $\pol$, naturally in $G$:
\[
\begin{tikzcd}
\Alg_k \arrow[r,"G"] \ar[dr,swap,"\pol"] & \Ab\\
& \Pol_p(k) \ar[u,"\tilde G"]\\
\end{tikzcd}
\]
\end{customthm}
Note the restriction to $p$-adic group schemes. There is also a companion result for formal groups where, interestingly, a similar restriction is not necessary:
\begin{customthm}{F}\label{thm:formalschemefactorization}
Let $k$ be a perfect field of characteristic $p$ and $G \in \FGps{k}$. Then the functor of points of $G$ factors through $\pol$, naturally in $G$:
\[
\begin{tikzcd}
\alg_k \arrow[r,"G"] \ar[dr,swap,"\pol"] & \Ab\\
& \pol_p(k) \ar[u,"\tilde G"]\\
\end{tikzcd}
\]
\end{customthm}

To prove Theorems~\ref{thm:affineschemefactorization} and \ref{thm:formalschemefactorization}, we make use of free (formal) group scheme functors:

\stepcounter{equation}
\begin{customlemma}{A}\label{lemma:freepadicgroup}
The forgetful functor $U\colon \AbSch{k}^p \to \Alg_k^{\op}$ from $p$-adic affine groups to $k$-algebras has a left adjoint $\Fr$.
\end{customlemma}
\begin{customlemma}{F}\label{lemma:freeformalgroup}
The forgetful functor $U\colon \FGps{k} \to (\Pro-\alg_k)^{\op}$ from formal groups to pro-finite dimensional $k$-algebras has a left adjoint $\Fr$.
\end{customlemma}


Using these free functors we show:

\stepcounter{equation}
\begin{customthm}{A}\label{thm:freeaffinegroupfactorization}
Let $k$ be a perfect field of characteristic $p$. Then $\Fr$ factors through $\pol$:
\[
\begin{tikzcd}
\Alg_k^\op \arrow[r,"\Fr"] \ar[dr,swap,"\pol"] & \AbSch{k}^p\\
& \Pol_p(k)^\op \ar[u,"\tilde \Fr"]\\
\end{tikzcd}
\]
\end{customthm}

\begin{customthm}{F}\label{thm:freeformalgroupfactorization}
Let $k$ be a perfect field of characteristic $p$. Then $\Fr$ factors through $\pol$:
\[
\begin{tikzcd}
\alg_k^\op \arrow[r,"\Fr"] \ar[dr,swap,"\pol"] & \FGps{k}\\
& \pol_p(k)^\op \ar[u,"\tilde \Fr"]\\
\end{tikzcd}
\]
\end{customthm}
In the latter theorem, the functor $\Fr$ is restricted to the subcategory $\alg_k$ of $\Pro-\alg_k$.

Assuming these theorems, we immediately obtain:

\begin{proof}[Proofs of Theorems~\ref{thm:affineschemefactorization} and \ref{thm:formalschemefactorization}]
Given any $M \in \AbSch{k}^p$ and $R \in \Alg_k$ (resp.  $M \in \FGps{k}$ and $R \in \alg_k$), we have that
\[
M(R) = \Hom(\Spec R, M) = \Hom(\Fr(R),M),
\]
where the last Hom group is of objects of $\AbSch{k}^p$ or $\FGps{k}$, respectively. Since $\Fr$ factors through $\Pol_p(k)$ (resp. $\pol_p(k)$), so does $M$.
\end{proof}

In terms of Hopf algebras, this can be reformulated in the following way. Theorem~\ref{thm:freeaffinegroupfactorization} says that the cofree cocommutative $p$-adic Hopf algebra functor on $k$-algebras factors through $p$-polar algebras. Using the Cartier equivalence between formal groups and bicommutative Hopf algebras, Theorem~\ref{thm:freeformalgroupfactorization} says that the \emph{free commutative} Hopf algebra functor on finite-dimensional $k$-coalgebras factors through the opposite category of finite-dimensional $p$-polar $k$-algebras.

\medskip
Instead of trying to prove Thms.~\ref{thm:freeaffinegroupfactorization} and \ref{thm:freeformalgroupfactorization} directly, we take a detour along Dieudonné functors to the land of Witt vectors. The version of Dieudonné functors we are using (cf. Thm.~\ref{thm:dieudonne}) are of the form 
\[
D\colon (\AbSch{k})^\op \to \DMod{k}^p
\]
and
\[
\DieuFor\colon (\FGps{k})^\op \to \DDModF{k},
\]
and define contravariant equivalences between $\AbSch{k}^p$ and $\FGps{k}^p$ and certain categories of $W(k)$-modules with two operations $F$ and $V$, called Frobenius and Verschiebung. Here $W(k)$ denotes the ring of $p$-typical Witt vectors of the field $k$. More generally, let $W_n(R)$ be the group of $p$-typical Witt vectors of length $n$ of a ring $R$. 

The Verschiebung  $V\colon W_n(R) \to W_{n+1}(R)$ gives rise to the group of unipotent co-Witt vectors
\[
CW^u(R) = \colim (W_1(R) \xrightarrow{V} W_2(R) \xrightarrow{V} \cdots),
\]
an object of $\DMod{k}^p$. It has a completion, $CW(R)$, the group of co-Witt vectors, consisting of possibly infinite negatively graded sequences $(\dots,a_{-1},a_0)$ of elements of $R$ almost all of which are nilpotent.

We prove:
\begin{thm}\label{thm:Wittofpolar}
Let $k$ be a perfect field. Then the functors $W_n$, $CW$, and $CW^u$ from $\Alg_k$ to $\DMod{k}^{p}$ factor naturally through $\Pol_p(k)$.
\end{thm}

The following couple of theorems provide the link betweeen co-Witt vectors and free (formal) groups:
\stepcounter{equation}
\begin{customthm}{F} \label{thm:dieudonneoffreeformal}
Let $k$ be a perfect field of characteristic $p$ and $R$ a finite-dimensional $k$-algebra. Then there is a natural isomorphisms
\[
\DieuFor(\Fr(R)) \cong CW(R).
\]
\end{customthm}
Since the right hand side factors through $\pol_p(k)$ by Thm.~\ref{thm:Wittofpolar}, so does the left hand side. This almost proves Thm.~\ref{thm:freeformalgroupfactorization} -- but not quite, since $\DieuFor$ only provides an equivalence of $\DDModF{k}$ with the full subcategory $\FGps{k}^p$. 

To state the affine version of this theorem, let $\mu_{p^\infty}(R)$ denote the abelian group of $p$-power torsion elements in $R^\times$, for a $k$-algebra $R$.
\begin{customthm}{A} \label{thm:dieudonneoffreeaffine}
Let $k$ be a perfect field of characteristic $p$ and $R$ a $k$-algebra. Then there is a natural isomorphism
\[
D(\Fr(R)) \cong CW^u(R) \oplus \left(\mu_{p^\infty}(R\otimes_k \bar k) \otimes W(\bar k)\right)^{\Gal(k)}
\]
where in the last factor, invariants of the absolute Galois group $\Gal(k)$ acting diagonally on $\bar k$ and $W(\bar k)$ are taken.
\end{customthm}
Again, Thm.~\ref{thm:freeaffinegroupfactorization} follows almost from this theorem; the missing piece in this case is to show that the second summand on the right hand side factors through $p$-polar algebras.

\subsection*{Overview}

In Section~\ref{sec:p-polar}, we define $p$-polar $k$-algebras and their properties. Section~\ref{sec:witt} develops the basic theory of $p$-adic Witt vectors for $p$-polar rings and contains a proof of Thm.~\ref{thm:Wittofpolar}. Section~\ref{sec:free} contains a review of the structure of the categories $\AbSch{k}$ and $\FGps{k}$ along with the proofs of Lemmas~\ref{lemma:freepadicgroup} and \ref{lemma:freeformalgroup}. Finally, Section~\ref{sec:dieudonne} contains the setup of the Dieudonné functors and the proofs of Thms.~\ref{thm:dieudonneoffreeformal}, \ref{thm:dieudonneoffreeaffine}, \ref{thm:freeaffinegroupfactorization}, and \ref{thm:freeformalgroupfactorization}.

\subsection*{Acknowledgement}
I would like to thank the anonymous referee for exceptionally constructive and encouraging feedback. Without their criticism, this paper would have been much less readable.

\section{\texorpdfstring{$p$-polar rings}{p-polar rings}}\label{sec:p-polar}
Recall the definition of a $p$-polar algebra $A$ from the introduction. We make the following observations:

\begin{itemize}
\item The definition is non-unital in nature. If one were to require the existence of an element $1 \in A$ such that $\mu(1,\dots,1,x)=x$, this would make $A$ into a commutative unital ring.
\item If $p=2$ then $A$ is a nonunital algebra.
\item The expression in (ASSOC) is $(\Sigma_p \times \Sigma_{p-1})$-equivariant by definition. Given commutativity, the condition is akin to an associative law. 
\end{itemize}

\begin{example}
The multiplication on an ordinary (not necessarily unital) ring $A$ restricts to polar ring structure $\pol(A)$. 
\end{example}

\begin{example}
We have that $k[x]_{(j)} =_{\text{def}} k\langle x^{j(1+(p-1)i)} \mid i \geq 0\rangle$ is a polar subalgebra of $k[x]$ for each $j \geq 0$, and it  is not an algebra (unless $j=0$, in which case it is just $k$, or $j>0$ and $p=2$, in which case it is an algebra without unity).
\end{example}

\begin{example}
Let $A = xk[x]/(x^p)$ and $B = k^{p-1}$ nonunital with trivial multiplication. Then $\mu=0$, and hence $\pol(A) \cong \pol(B)$ as $p$-polar algebras. 
\end{example}

\begin{remark}\label{remark:unity}
For a unital algebra $A$, one can recover $A$ from $\pol(A)$ up to non-canonical isomorphism. Indeed, if $B$ is a $p$-polar algebra of the form $\pol(A)$, there is an element $e$ such that $\mu(e,\dots,e,x)=x$ for all $x \in B$. For instance, the unity of $A$ is such an element, but $B$ does not preserve that information. One easily checks that an algebra structure on $B$ can be defined by $x \cdot y = \mu(e,\dots,e,x,y)$. With this algebra structure, multiplication by $e$ gives an algebra isomorphism $A \to B$.

For nonunital algebras, this is not possible, as the previous example illustrates.
\end{remark}

\begin{prop}\label{prop:onlyonemultiplication}
Let $A$ be a $p$-polar $k$-algebra and $x_1,\dots,x_n \in A$. Then there is at most one way to multiply $x_1,\dots,x_n$ together using $\mu$, and the product exists if and only if $n \equiv 1 \pmod {p-1}$.
\end{prop}

\begin{proof}
Let us first make the statement more rigorous. Define a \emph{multiplication scheme} to be a rooted, planar, $p$-ary tree whose leaves are labelled with $x_1,\dots,x_n$. If a multiplication scheme exists, then it gives a prescription of how to multiply the elements $x_i$ by traversing the tree, applying $\mu$ at every internal vertex. If no multiplication scheme exists, then the elements cannot be multiplied together. Since grafting a basic $p$-ary tree of depth $1$ onto an existing tree increases the number of leaves by $p-1$, a multiplication scheme exists iff $n \equiv 1 \pmod {p-1}$.

An equivalence relation on the set of all multiplication schemes is generated by:
\begin{enumerate}
	\item $M' \sim M$ if $M'$ results from $M$ by permuting the outgoing edges of any internal vertex;
	\item $M' \sim M$ if $M'$ results from $M$ by permuting the $2p-1$ outgoing edges of a subtree of the form, along with their subtrees (illustrated for $p=3$).
	
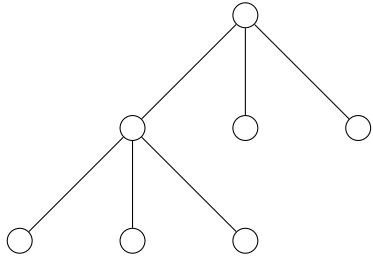
\[
	\begin{tikzpicture}[every node/.style = {shape=circle,draw,align=center}]
	\node {}
	 child { node {}
	 	child { node {}}
	 	child { node {}}
	 	child { node {}}
	 	}
	 child { node {}}
	 child { node {}};
	 \end{tikzpicture}
	\]
\end{enumerate}
These relations correspond, of course, to the symmetry and axiom (ASSOC). Under this equivalence relation, every multiplication scheme is equivalent to a ``left-associative'' one, i.e. a scheme where a vertex has a non-leaf subtree only if it is the leftmost among its siblings, such as in the diagram above. In such a left-associative multiplication scheme, all leaves can be permuted without changing the equivalence class. Thus all multiplication schemes are equivalent.
\end{proof}

In light of this proposition, we will unambiguously use monomial notations such as $x^p$ for $\mu(x,\dots,x)$ or $xy^{p-1}$ for $\mu(x,y,\dots,y)$ in polar rings.

\begin{example}\label{ex:free-p-polar}
Let $S$ be a set. The free $p$-polar ring $P(S)$ on $S$ is given by the sub-$p$-polar algebra of the polynomial ring $\Z[S]$ with generators in $S$ spanned by monomials of length congruent to $1$ modulo $p-1$ (or spanned by all nonconstant monomials if $p=2$).
\end{example}

\begin{numbereddefn} \label{defn:ideal}
An \emph{ideal} in a $p$-polar ring $A$ is a subgroup $I$ such that $\mu(a_1,\dots,a_{p-1},i) \in I$ whenever $i \in I$. These are exactly the kernels of homomorphisms of $p$-polar rings. For a subset $S \subseteq A$, the ideal $(S)$ generated by it is defined to be the smallest ideal containing $S$. If $I$ is an ideal, then $I^p=\langle \mu(I,\dots,I)\rangle$ is a subideal.
\end{numbereddefn}

\section{\texorpdfstring{Witt vectors of $p$-polar rings}{Witt vectors of p-polar rings}}\label{sec:witt}

The ($p$-typical) Witt vector functor \cite{witt:witt-vectors} $W$ and its truncated variants $W_n$ take values in rings and are defined on the category of rings. Since $W_1(A) \cong A$, no information of the input ring is lost. However, if one is only interested in $W(A)$ as an abelian group with Frobenius and Verschiebung operations, one can ask what the minimal required structure on $A$ is. A ring structure is enough; an abelian group structure alone is not. It turns out that the structure is exactly that of a $p$-polar ring. Background on Witt vectors and related constructions can be found in \cite{hesselholt:witt-survey,hazewinkel:witt,serre:corps-locaux}.

\begin{defn}
For $0 \leq n \leq \infty$ and a $p$-polar ring $A$, define its set of Witt vectors by
\[
W_n(A) = \prod_{i=0}^{n-1} A \quad \text{(and $W(A)=W_\infty(A)$.)}
\]
\end{defn}

Just as in the classical case, there is a ghost map
\begin{equation}\label{eq:ghost}
w\colon W_n(A) \to \prod_{i=0}^{n-1} A \qquad (0 \leq n \leq \infty)
\end{equation}
given by
\[
w(a_0,a_1,\dots) = (a_0,a_0^p+pa_1,a_0^{p^2}+pa_1^p+p^2a_2,\dots).
\]

We will make use of the following version of Dwork's lemma for polar rings:
\begin{lemma}[Dwork] \label{lemma:dwork}
Let $A$ be a $p$-polar ring and $0 \leq n \leq \infty$. Assume that there is a polar ring map $\phi\colon A \to A$ such that $\phi(a) \equiv a^p \pmod{p}$. Then a sequence $(x_0,x_1,\dots) \in \prod_{i=0}^{n-1} A$ is in the image of $w$ iff $x_m \equiv \phi(x_{m-1}) \pmod{p^m}$ for all $m \geq 1$.
\end{lemma}
The classical proof, e.g. as in \cite[Lemma~1]{hesselholt:witt-survey}, works almost without changes. For the reader's convenience, we include it here.

\begin{proof}
As a first step, we show that $a^{p} \equiv b^{p} \pmod{p^{m+1}}$ if $a \equiv b \pmod{p^m}$. Indeed, if $b=a+p^m d$ then
\[
b^p = (a+p^m d)^p = a^p + \sum_{i=1}^p \binom{p}{i} p^{mi} a^{p-i} d^i
\]
by multilinearity and symmetry of $\mu$. Since $\binom{p}{1}=p$, then $p^{m+1} \mid b^p-a^p$.

In particular, $\phi(a)^{p^m} \equiv a^{p^m} \pmod {p^{m+1}}$ by induction. Now, since $\phi$ is a homomorphism of polar rings, we have that
\[
\phi(w_{m-1}(a)) = \sum_{i=0}^{m-1} p^i \phi(a_i)^{p^{m-1-i}} \equiv \sum_{i=0}^m p^i a_i^{p^{m-i}} = w_m(a)\pmod{p^m}.
\]
This shows that if $(x_i)_{0 \leq i \leq n}$ is in the image of $w$ then it satisfies the stated congruence. Conversely, suppose the congruence holds. Construct $a_i$ inductively by first choosing $a_0=x_0$. Having constructed $a_0,\dots,a_{m-1}$, observe that
\[
D_m = x_m - \sum_{i=0}^{m-1} p^i a_i^{p^{m-i}} \equiv 0 \pmod{p^m}.
\]
Thus choose $a_m$ such that $p^m a_m = D_m$.
\end{proof}

\begin{lemma}\label{lemma:Wittofppolar}
Let $A$ be a $p$-polar ring and $1 \leq n \leq \infty$.
\begin{enumerate}
\item There is a unique natural $p$-polar ring structure on $W_n(A)$ making the ghost map \eqref{eq:ghost} into a $p$-polar ring map. \label{lemma:Wittofppolar:ppolarstructure}
\item \label{lemma:Wittofppolar:FV} There are unique natural additive maps $F\colon W_{n+1}(A) \to W_n(A)$ (Frobenius) and $V\colon W_{n}(A) \to W_{n+1}(A)$ (Verschiebung) such that the following diagram commutes ($x_{-1}=0$ by convention):
\[
\begin{tikzcd}[column sep={10em,between origins}]
W_{n}(A) \arrow[d,"w"] \arrow[r,"V"] & W_{n+1}(A) \arrow[d,"w"] \arrow[r,"F"] & W_{n}(A) \arrow[d,"w"]\\
\prod_{i=0}^{n-1} A \arrow[r,"(x_i)_i \mapsto (px_{i-1})"] &
\prod_{i=0}^{n} A \arrow[r,"(x_i)_i \mapsto (x_{i+1})_i"] & \prod_{i=0}^{n-1} A
\end{tikzcd}
\]
Explicitly, $V(a_0,\dots,a_{n-1}) = (0,a_0,\dots,a_{n-1})$ and $F$ is uniquely determined by $F(\underline a) = \underline{a^p}$ and $FV=p$, where $\underline a = (a,0,\dots,0)$ denotes the Teichm\"uller representative of $a \in A$ in $W_n(A)$.
\item \label{lemma:Wittofppolar:algebra} If $A$ is a $p$-polar algebra over a perfect field $k$ of characteristic $p$ then $W_n(A)$ is a $p$-polar $W(k)$-algebra. Denoting the Frobenius $F$ on $W(k)$ by $\frob$ to distinguish it from the Frobenius on $W_n(A)$, we have that
\begin{equation}\label{eq:FVscalarcommutation}
 Fa=\frob(a)F \quad \text{and} \quad Va=\frob^{-1}(a)V \quad \text{for $a \in W(k)$}.
 \end{equation}
\end{enumerate}
\end{lemma}
Note that $\frob$ is bijective since $k$ is perfect.
\begin{proof}
For \eqref{lemma:Wittofppolar:ppolarstructure}, the classical proof (cf. \cite[Prop.~2]{hesselholt:witt-survey}) works. First consider the free $p$-polar ring $P = P(a_{i,j} \mid 0 \leq i < n, 1 \leq j \leq p)$ of Ex.~\ref{ex:free-p-polar} spanned by polynomials of degree congruent to $1$ modulo $p-1$ and the $p$-polar ring map $\phi\colon P \to P$ with $\phi(a_{i,j})=a_{i,j}^p$. Write $a_j$ for the  sequence $( a_{i,j} \mid 0 \leq i < n)$. By Lemma~\ref{lemma:dwork}, $w(a_1)+w(a_2)$, $-w(a_1)$, and $w(a_1)w(a_2)\cdots w(a_p)$ are in the image of $w$. Since $P$ is torsion free, $w$ is injective. Thus this defines a unique $p$-polar ring structure on $W_n(P)$. The case for arbitrary $A$ follows from naturality.
Assertions \eqref{lemma:Wittofppolar:FV} and \eqref{lemma:Wittofppolar:algebra} follow from similar arguments, considering the universal, torsion-free case first and using naturality to deduce the general case.
\end{proof}

Now let $k$ be a perfect field of characteristic $p$. Denote by $\mathcal R = W(k)\langle F,V\rangle/(FV-p)$ the noncommutative ring obtained from $W(k)$ by adjoining two variables $F$, $V$ such that $FV=VF=p$ and such that commutation with scalars is governed by \eqref{eq:FVscalarcommutation}. Then for a $p$-polar $k$-algebra $A$ and $n \leq \infty$, $W_n(A)$ becomes an $\mathcal R$-module  by Lemma~\ref{lemma:Wittofppolar}.

As in \cite[Ch. II]{fontaine:groupes-divisibles} or \cite[\textsection 5.3]{bauer-carlson:tensorproduct}, we define the group of unipotent co-Witt vectors as the colimit
\[
CW^u(A) = \colim(W_0(A) \xrightarrow{V} W_1(A) \xrightarrow{V} \cdots)
\]
This works when $A$ is merely a $p$-polar ring by Lemma~\ref{lemma:Wittofppolar}.

Define the set of co-Witt vectors as
\[
CW(A) = \{(a_i) \in A^{\Z_{\leq 0}}\mid (a_{-r},a_{-r-1},\dots)^{p^s}=0 \text{ for some } r,s \geq 0\},
\]
where $(a_{-r},a_{-r-1},\dots)$ denotes the ideal (cf. Def.~\ref{defn:ideal}) in $A$ generated by the given elements.
This functor (as a functor on finite-dimensional $k$-algebras) is in fact a formal group \cite[\textsection II.4]{fontaine:groupes-divisibles}, i.e. ind-representable.

To see that $CW(A)$ has the structure of an $\mathcal R$-module even when $A$ is just a $p$-polar ring, containing $CW^u(A)$ as a submodule, we use and adapt the arguments of \cite[\textsection II.1.5]{fontaine:groupes-divisibles}.

Define the polynomial $S_n(x_0,\dots,x_n,y_0,\dots,y_n)$ as the $n$th component of the Witt vector addition $(x_0,x_1,\dots,x_n) + (y_0,y_1,\dots,y_n) \in W(\Z[x_0,\dots,x_n,y_0,\dots,y_n])$. Note that this polynomial is in fact an element of the free $p$-polar ring (cf. Ex.~\ref{ex:free-p-polar}) $P(x_1,\dots,x_n,y_1,\dots,y_n)$, and hence can be evaluated on elements of $p$-polar rings.

\begin{prop}
Let $A$ be a $p$-polar ring and $a=(a_i)$, $b=(b_i) \in CW(A)$.  Then
\begin{enumerate}
	\item For each $n \geq 0$, the sequence
	\[
	S_m(a_{-n-m},\dots,a_{-n},b_{-n-m},\dots,b_{-n})
	\]
	 is eventually constant as $m \to \infty$. Let us call the limit value $S_{-n}(a,b)$.
	 \item $S_{-n}(a,b) \in CW(A)$.
\end{enumerate}
\end{prop}
\begin{proof}
This is proved for commutative rings in \cite[Prop. II.1.1]{fontaine:groupes-divisibles}.

Since $a$, $b \in CW(A)$, there exist natural numbers $r,s$ such that
\[
(a_{-r},a_{-r-1},\dots,b_{-r},b_{-r-1},\dots)^{p^s}=0. 
\]
Let
\[
R=\Z[x_i,y_i \mid i \leq 0]/(x_{-r},x_{-r-1},\dots,y_{-r},y_{-r-1},\dots)^{p^s}.
\]
Then by \cite[Prop. II.1.1]{fontaine:groupes-divisibles}, the sequence $S_m(x_{-n-m},\dots,x_{-n},b_{-n-m},\dots,b_{-n})$ is eventually constant as $m \to \infty$, and its limit lies in $CW(R)$. In fact, as observed before, it lies in $CW(P)$, where
\[
P=P(x_i,y_i \mid i \leq 0]/(x_{-r},x_{-r-1},\dots,y_{-r},y_{-r-1},\dots)^{p^s} < R.
\]
Since this is the universal case for elements $a$, $b$ with the chosen vanishing properties, the claim follows from naturality.
\end{proof}

The $\mathcal R$-module structure on $CW(A)$ is thus given by
\begin{align*}
a+b &= (\dots,S_{-2}(a,b),S_{-1}(a,b),S_0(a,b)),\\
Va &= (\dots,a_2,a_1), \quad \text{and}\\
Fa &= (\dots,a_2^p,a_1^p,a_0^p).
\end{align*}

\begin{remark}
The structure of a $p$-polar algebra is the minimal structure needed to define the Witt vector functor; indeed, the $p$-polar $k$-algebra $A$ can be reconstructed from the abelian groups $W_1(A)$ and $W_2(A)$ together with the Teichmüller map $t\colon A \to W_2(A)$, $t(a)=\underline a = (a,0)$ and the Verschiebung as follows. As an abelian group, $A = W_1(A)$. For elements $x_1,\dots,x_p \in A$, we have:
\[
\sum_{I \subseteq \{1,\dots,p\}} (-1)^{\#I} t\Bigl(\sum_{i \in I} x_i\Bigr) = (0,x_1\cdots x_p) \in W_2(A),
\]
so the $p$-polar structure can be extracted by equating $V(\tau(\mu(x_1,\dots,x_p))$ to the above sum. The reader may amuse themself by deriving the above equality from the inductively proved identity
\[
\sum_{I \subseteq \{1,\dots,k\}} (-1)^{\#I} t\Bigl(\sum_{i \in I} x_i\Bigr) = \Bigl(0,\sum_{\substack{i_1+\cdots+i_k = p\\i_1,\dots,i_k\geq 1}} \frac 1p \binom{p}{i_1,\dots,i_k} x_1^{i_1}\cdots x_k^{i_k}\Bigr).
\]
for $k \geq 2$.
\end{remark}

\begin{example}[Witt vectors of free $p$-polar algebras]
The ring $W(\F_p[x])$ is described in \cite[Exercise 1(10)]{borger:witt-vector-lectures}. It is a subring of the $p$-completed monoid ring $\Z_p[\mathbf N[\frac 1 p]]\hat{{}_p} = W(\F_p[\mathbf N[\frac 1 p]])$ consisting of those series $\sum a_{\frac n {p^i}} [\frac n {p^i}]$ such that $p^i \mid a_{\frac n {p^i}}$. The group $W(P(x))$ of the free $p$-polar algebra on one generator is the subgroup of power series where $a_{\frac n {p^i}}=0$ unless $n \equiv 1 \pmod{p-1}$.
\end{example}

\section{Free affine and formal groups} \label{sec:free}

Continue to let $k$ be a perfect field of characteristic $p$. A formal group $G$ is called \emph{connected} if its representing pro-$k$-algebra $\Reg{G}$ is a pro-local $k$-algebra, and it is called \emph{\'etale} if $\Reg{G}$ is pro-\'etale as a pro-$k$-algebra. An affine group $G$ is called \emph{unipotent} if its Cartier dual $G^*$ is connected, and \emph{of multiplicative type} if $G^*$ is \'etale.

The following splittings are classical:
\begin{thm}[{\cite[\textsection I.7]{fontaine:groupes-divisibles}}] \label{thm:splitting}
The category $\FGps{k}$ splits naturally as $\FGpse{k} \times \FGpsc{k}$, where $\FGpse{k}$ and $\FGpsc{k}$ denote the full subcategories of \'etale resp. connected formal groups.

Dually, the category $\AbSch{k}$ splits naturally as $\AbSchm{k}$ and $\AbSchu{k}$ into the product of its full subcategories of multiplicative-type and of unipotent groups.
\end{thm}

Recall that an affine group is $p$-adic if it takes values in abelian pro-$p$-groups, and that a formal group is a formal $p$-group if it takes values in $p$-groups.

\begin{remark} \label{rem:unipotentpadic}
If $G$ is a unipotent affine group then $G$ is automatically $p$-adic; dually, if $G$ is a connected formal group then $G$ is automatically a formal $p$-group. Since these two statements are Cartier dual to one another, it suffices to show the first. An affine group is unipotent if and only if its representing Hopf algebra $H$ is \emph{conilpotent}. In particular, for each $x \in H$, there exists an $n \geq 0$ such that $V^n(x)=0$, where $V$ denotes the Hopf algebra Verschiebung. Since $[p]=FV=VF$ in any abelian Hopf algebra, with $F$ the Hopf algebra Frobenius, we have that $x \in H[p^n] = \ker([p^n])$. Thus $H \cong \colim_n H[p^n]$, and $H$ is $p$-adic.
\end{remark}

We can now show that the free $p$-adic affine group functor $\Fr$ exists:
\begin{proof}[Proof of Lemma~\ref{lemma:freepadicgroup}]
The forgetful functor from $\AbSch{k}^p$ to $\Alg_k^\op$ factors as
\[
\AbSch{k}^p \overset{J}\hookrightarrow \AbSch{k} \xrightarrow{U} \Alg_k^\op,
\]
and hence it suffices to construct left adjoints for $J$ and $U$. Since an affine group $G$ is $p$-adic iff $G \cong G\hat{{}_p} = \lim_n G/p^n$ as affine group schemes, the functor $J$ has the $p$-completion $G \mapsto G\hat{{}_p}$ as a left adjoint. On the level of Hopf algebras, the adjoint to the functor $U$ corresponds to the cofree cocommutative Hopf algebra on an algebra $A$, which was constructed in \cite[Proof of Thm.~1.3]{bauer-carlson:tensorproduct}. \end{proof}

Let us be more explicit. By Thm.~\ref{thm:splitting}, $\Fr \cong \Fr^m \times \Fr^u$ splits into a part of multiplicative type and a unipotent part. Since unipotent groups are $p$-adic (Rk.~\ref{rem:unipotentpadic}), $\Fr^u$ is represented by the cofree cocommutative conilpotent Hopf algebra on $R$, first constructed by Takeuchi \cite{takeuchi:tangent-coalgebras}. It is the Hopf algebra of symmetric tensors $\bigoplus_{n \geq 0} (R^{\otimes_k n})^{\Sigma_n}$.

A group of multiplicative type is $p$-adic if it is isomorphic to $\Spec \bar k[M]$ for a $p$-torsion group $M$ after base change to an algebraic closure $\bar k$ of $k$. In the case where $k \cong \bar k$, the free multiplicative $p$-adic group on $\Spec A$ is represented by $k[\mu_{p^\infty}(A)]$, and in the general case, it is represented by the Galois invariants (cf.~\cite[Proof of Thm. 1.3]{bauer-carlson:tensorproduct})
\[
\bar k[\mu_{p^\infty}(R\otimes \bar k)]^{\Gal(k)}.
\]  

The proof of Lemma~\ref{lemma:freeformalgroup} follows similar lines.
\begin{proof}[Proof of Lemma~\ref{lemma:freeformalgroup}]
Taking $k$-linear continuous duals of representing objects gives a commutative diagram
\[
\begin{tikzcd}
\FGps{k} \arrow[r,"U"] \arrow[d,"G \mapsto \Reg{G^*}"] & (\Pro-\alg_k)^\op \arrow[d,"(-)^*"]\\
\Hopf{k} \arrow[r,"U_H"] & \Coalg_k,
\end{tikzcd}
\]
where $\Coalg_k$ denotes the category of cocommutative coalgebras over $k$ and $U_H$ is the forgetful functor. The left adjoint of $U_H$ is the free commutative Hopf algebra on a cocommutative coalgebra and is constructed in \cite{takeuchi:free-hopf}.
\end{proof}

Again, let us be more explicit. By Thm.~\ref{thm:splitting}, there is a natural splitting $\Fr(R) \cong \Fr^e(R) \times \Fr^c(R)$ into an \'etale and a connected part.

Let us describe the $\Fr^e(R)$ more concretely, for a finite-dimensional $k$-algebra $R$. The category of étale formal groups is equivalent to the category $\Mod_\Gamma$ of abelian groups with a discrete action of the absolute Galois group $\Gamma = \Gal(k)$ (cf. \cite[\textsection I.7]{fontaine:groupes-divisibles}, \cite[Thm.~1.6]{bauer-carlson:tensorproduct}); this equivalence is simply given by the functor $G \mapsto \colim_{k < k' < \bar k} G(k')$, where $k'$ runs through the finite extensions of $k$, and the inverse functor is given by $M \mapsto \Spf{\map^\Gamma(M,\bar k)}$. 

\begin{lemma}\label{lemma:freeetale}
Let $R$ be a finite-dimensional $k$-algebra. Then the $\Gamma$-module associated with $\Fr^e(R)$ is 
\[
\Z\langle\Hom_{\Alg_k}(R,\bar k)\rangle
\]
\end{lemma}

I thank the anonymous referee for suggesting the following simplified proof: 
\begin{proof}
Consider the following $2$-commutative diagram of functors:
\[
\begin{tikzcd}
\FGpse{k} \arrow[r,"U"] \arrow[d] & (\alg_k^e)^\op \arrow[d]\\
\Mod_\Gamma \arrow[r,"U_\Gamma"] & \Set_\Gamma
\end{tikzcd},
\]
where $\alg_k^e$ denotes the category of finite-dimensiona, \'etale $k$-algebras, $\Set_\Gamma$ the category of discrete $\Gamma$-sets, and $U_\Gamma$ the forgetful functor. The vertical arrows are equivalences and given by taking $\bar k$-valued points. Since the left adjoint of $U_\Gamma$ is the free abelian group functor on a $\Gamma$-set $X$, the result follows.
\end{proof}

\section{The Dieudonné correspondence and proof of the main theorems}\label{sec:dieudonne}

Continue to let $k$ be a perfect field of characteristic $p$.

\begin{numbereddefn} \label{def:discretedieudonnemodule}
Denote by $\DMod{k}$ the category of Dieudonné modules over $k$. These are $\mathcal R$-modules, i.e. $W(k)$-modules with homomorphisms
\[
V\colon M \to M \quad \text{and} \quad F\colon M \to M
\]
such that $FV=VF=p$, $Fa = \frob(a)F$ and $aV = V\frob(a)$ for $a \in W(k)$.
\end{numbereddefn}

Denote by $\DMod{k}^p$ the full subcategory of \emph{$p$-adic} Dieudonné modules. These are modules $M$ such that for all $x \in M$, the $W(k)$-submodule spanned by $V^i(x)$, for all $i$, is of finite length.

Dually, let $\DDModF{k}$ be the full subcategory of \emph{$F$-profinite} Dieudonné modules, i.~e. modules $M$ that are profinite as $W(k)$-modules and have a fundamental system of neighborhoods consisting of $W(k)$-modules closed under $F$.

The following theorem follows from classical Dieudonné theory (cf. \cite[Section 6]{bauer-carlson:tensorproduct}, \cite{demazure-gabriel:groupes-algebriques-1}):

\begin{thm}\label{thm:dieudonne}
There are equivalences of abelian categories
\[
D \colon (\AbSch{k}^p)^\op \to \DMod{k}^p
\]
and
\[
\DieuFor\colon (\FGps{k}^p)^\op \to \DDModF{k}
\]
These functors are given by
\[
D(G) = \colim_n \Hom_{\AbSch{k}}(G,W_n) \oplus \Bigl(\Gr(\Reg{G} \otimes_k \bar k) \otimes W(\bar k)\Bigr)^\Gamma
\]
and
\[
\DieuFor(G) = \Hom_{\FGps{k}}(G,CW).
\]
Here, $\Gr(H)$ stands for the group of group-like elements of a Hopf algebra $H$.
\end{thm}

The formal part of this theorem is \cite[Thm.~6.5]{bauer-carlson:tensorproduct}, and the affine part is \cite[Thm.~6.3]{bauer-carlson:tensorproduct}. In the latter, if $G \cong G^u \times G^m$ is the splitting into a unipotent part and a part of multiplicative type (Thm.~\ref{thm:splitting}), the first summand of $D(G)$ is $D(G^u)$ and the second is $D(G^m)$. The statement in \cite[Thm.~6.3]{bauer-carlson:tensorproduct} actually contains an error for $D(G^m)$, and we take the opportunity to rectify it here.

\begin{proof}[Correction of {\cite[Thm.~6.3]{bauer-carlson:tensorproduct}}]
The proof of $D(G^u) = \colim_n \Hom_{\AbSch{k}}(G,W_n)$ is contained in \cite[Chapter III, 6]{demazure:pdivisiblegroups}. Thus let $G$ be of multiplicative type, and let
\[
D'(G) = I(\DieuFor(G^*)),
\]
where $G^*$ denotes the formal group Cartier dual to $G$, and $I=\Hom_{W(k)}(-,CW(k))$ denotes Matlis (or Pontryagin) duality between $W(k)$-modules and pro-(finite length) $W(k)$-modules. Since both dualities are anti-equivalences of categories and $\DieuFor$ is an equivalence, so is $D'$. It remains to show that $D'(G) = D(G)$.

For this, recall (e.g. from \cite[\textsection I.7]{fontaine:groupes-divisibles}) that the category of affine groups of multiplicative type over $k$ is equivalent with the category $\Mod_\Gamma$ of abelian groups with a discrete action of the absolute Galois group $\Gamma$ via the functor $G \mapsto \Gr(\Reg{G} \otimes_k \bar k)$.

Write $\overline{CW} = \colim_{k < k' < \bar k} CW(k')$ and $\overline W = \colim_{k < k' < \bar k} W(k')$, where $k'$ runs through all finite field extensions of $k$ contained in some algebraic closure $\bar k$. Then in terms of these $\Gamma$-modules, the functors $D$ and $D'$ are given, respectively, by
\[
D(M) = (M \otimes \overline W)^\Gamma
\]
and
\[
D'(M) = I(\Hom^{\Gamma}(M,\overline{CW})) = I(\Hom^{\Gamma}_{\overline W}(M \otimes \overline W,\overline{CW}))
\]
By Galois descent (cf. \cite[\textsection III.2]{fontaine:groupes-divisibles}), the category of $\overline W$-modules with semilinear $\Gamma$-actions is equivalent to the category of $W(k)$-modules by taking $\Gamma$-fixed points, and hence the last expression equals
\[
I(\Hom_{W(k)}((M \otimes \overline W)^\Gamma, CW(k)) = (M \otimes \overline W)^\Gamma.\qedhere
\]
\end{proof}

\begin{proof}[Proof of Thm.~\ref{thm:dieudonneoffreeformal}]
For a formal scheme $S$ represented by a profinite algebra $A$ (in particular, for algebras of finite dimension over $k$), we have that
\[
\DieuFor(\Fr(A)) = \Hom_{\FGps{k}}(\Fr(A),CW) = \Hom_{\FSch{k}}(\Spf A,CW) = CW(A).
\]
Here $\FSch{k}$, the category of formal schemes over $k$, is dual to the category of profinite $k$-algebras.
\end{proof}

\begin{proof}[Proof of Thm.~\ref{thm:dieudonneoffreeaffine}]
Let $S$ be an affine scheme $S$ represented by an algebra $A$. Then
\begin{align*}
D(\Fr(A)) =& \colim\Hom_{\AbSch{k}}(\Fr(A),W_n) \oplus \Bigl(\mu_{p^\infty}(A \otimes_k \bar k) \otimes W(\bar k)\Bigr)^\Gamma\\
=& \colim W_n(A) \oplus \Bigl(\mu_{p^\infty}(A \otimes_k \bar k) \otimes W(\bar k)\Bigr)^\Gamma. \qedhere
\end{align*}
\end{proof}

In order to prove Thm.~\ref{thm:freeaffinegroupfactorization}, we need so see that the multiplicative factor in the previous statement is well-defined for $p$-polar algebras. This will follow from the following result about $p$-typical formal group laws.

Recall (e.g. from \cite[A2.1.17]{ravenel:green}) that a $1$-dimensional formal group law $F$ over a torsion free $\Z_p$-algebra $R$ is \emph{$p$-typical} iff its logarithm $\log_F(x) \in (R \otimes \Q)\pow{x}$ is of the form $\log_F(x) = \sum_{i \geq 0} l_i x^{p^i}$ (with $l_0=1$). 

\begin{lemma}\label{lemma:ptypicalfactors}
Assume that $G \in \FGps{k}$ is the mod-$p$ reduction of a $1$-dimensional formal group over a torsion free $\Z_p$-algebra $R$. Then Theorem~\ref{thm:formalschemefactorization} holds for $G$.
\end{lemma}
\begin{proof}
Recall (e.g. from \cite[A2.1.17]{ravenel:green}) that a $1$-dimensional formal group law $F$ over a torsion free $\Z_p$-algebra $R$ is \emph{$p$-typical} iff its logarithm $\log_F(x) \in (R \otimes \Q)\pow{x}$ is of the form $\log_F(x) = \sum_{i \geq 0} l_i x^{p^i}$ (with $l_0=1$). By Cartier's Theorem \cite{cartier:formal-groups}, cf. \cite[Thm.~A2.1.18]{ravenel:green}, every one-dimensional formal group law is isomorphic to a $p$-typical one, so we may assume that $G$ is the mod-$p$ reduction of a $p$-typical formal group law $F$ over $R$. It is straightforward to see that the (compositionally) inverse power series $\exp_F(x)$ has the form
\[
\exp_F(x) = \sum_{i \geq 0} a_i x^{1+i(p-1)} \quad \text{for some } a_i \in R \otimes \Q,\; a_0=1.
\]
Thus both $\log_F$ and $\exp_F$ are in fact elements of the $p$-polar power series algebra
\[
\prod_{i \geq 0} (R\otimes \Q)\langle x^{1+i(p-1)} \rangle \subset (R \otimes \Q)\pow{x}.
\]
Hence if $A \in \pol_p(k)$, we can define
\[
\tilde G(A) = \nil(A)
\]
with the group structure
\[
x * y = \exp_F(\log_F(x) + \log_F(y)),
\]
agreeing with $G(R)$ if $A=\pol_p(R)$.
\end{proof}

\begin{corollary}\label{cor:punitsofpolar}
The functor $\mu_{p^\infty}\colon \Alg_k \to \Ab$ factors through $\Pol_p(k)$.
\end{corollary}
\begin{proof}
The functor
\[
\mu_{p^\infty}(R) = \{ x \in R^\times \mid x^{p^n}=1 \text{ for some } n \geq 0\}
\]
is the colimit of the functors $\mu_{p^n}(R)$ represented by $k[y]/(y^{p^n}-1) \cong k[x]/(x^{p^n})$ where $x=y-1$. Thus $\mu_{p^\infty(R)} \cong \nil(R)$ as sets, and the group structure on $\nil(R)$ is the multiplicative one: $x * y = x+y+xy$. Thus the restriction of $\mu_{p^\infty}$ to $\alg_k$ is isomorphic to the multiplicative formal group $\hat {\mathbb G}_m$. Since it is defined over $R=\Z_p$ and one-dimensional, Lemma~\ref{lemma:ptypicalfactors} applies and shows that $\mu_{p^\infty}$ factors through \emph{finite} $p$-polar $k$-algebras.

Now, for an arbitrary $p$-polar $k$-algebra $A$, define
\[
\tilde{\mu}_{p^\infty} = \colim_{B<A} \mu_{p^\infty}(B), 
\]
where $B$ ranges over the finite-dimensional $p$-polar subalgebras of $A$. Since $\nil(R) = \colim_{B<A} \nil(B)$ because finitely generated nilpotent subalgebras of $R$ are finite-dimensional, this is a factorization as required.
\end{proof}

\begin{proof}[Proof of Thm.~\ref{thm:freeaffinegroupfactorization}]
By Thm.~\ref{thm:Wittofpolar}, $CW^u\colon \Alg_k \to \DMod{k}^p$ factors through $\Pol_p(k)$, and by Cor.~\ref{cor:punitsofpolar}, so does $\mu_{p^\infty}\colon \Alg_k \to \Ab$. It follows that also the functor
\[
R \mapsto \left(\mu_{p^\infty}(R \otimes_k \bar k) \otimes W(\bar k)\right)^{\Gal(k)}\colon \Alg_k \to \DMod{k}^p
\]
appearing on the right hand side in Thm.~\ref{thm:dieudonneoffreeaffine}, factors through $\Pol_p(k)$. Thus, by the said theorem, $R \mapsto D(\Fr(R))$ factors through $\Pol_p(k)$. Since $D$ is an anti-equivalence between $\AbSch{k}^p$ and $\DMod{k}^p$, the result follows.
\end{proof}

To prove Thm.~\ref{thm:freeformalgroupfactorization}, we need to study the \'etale part $\Fr^e$ of the free formal group functor more closely.

\begin{lemma}\label{lemma:reducedstructure}
Let $k$ be algebraically closed and let $A$ be a finite, \emph{reduced} $p$-polar $k$-algebra, i.e. one such that for $x \neq 0$, we have that $x^{p^N}\neq 0$ for all $N\geq0$. Then $A$ is isomorphic to a finite product of polarizations of $k$.
\end{lemma}
\begin{proof}
Suppose $A$ has a unity in the sense of Remark~\ref{remark:unity}, i.e. an element $u\in A$ such that $u^{p-1}a=a$ for all $a \in A$. Then, by the quoted remark, $A = \pol(\tilde A)$, where $\tilde A$ is the algebra structure on $A$ given by $x\cdot y = \mu(u,\dots,u,x,y)$. Since $\tilde A$ is reduced and finite, it is a product of finitely many copies of $k$.

For the general case, I claim that $A$ has a nonzero element $e$ such that $e^p=e$. 

Choose a nonzero $y \in A$ and $j \geq 1$ such that the powers $y,y^p,\dots,y^{p^{j-1}}$ are linearly independent and $y^{p^j}=\sum_{i=0}^{j-1} \alpha_i y^{p^i}$ for some $\alpha_i \in k$. Such $y$ and $j$ must exist by the finiteness of $A$. Since $A$ is reduced, not all $\alpha_i$ are zero.
Let $\beta$ be a nonzero root of the polynomial
\[
p(x) = \sum_{i=0}^{j-1} \alpha_i^{p^{j-i-1}} x^{p^{j-i}} - x
\]
and let
\[
e = \sum_{l=0}^{j-1} \Bigl( \sum_{i=0}^l (\alpha_i\beta^p)^{p^{l-i}}\Bigr) y^{p^l}.
\] 
Then it is straightforward to verify that $e^p=e$. Since the $y^{p^l}$ are linearly independent, $e$ can only be zero if $\sum_{i=0}^l(\alpha_i\beta^p)^{p^{l-i}} = 0$ for all $l \leq j-1$. But this can only happen if all $\alpha_i=0$, contrary to the assumption.

Now consider the map $f\colon A \to A$ given by $y \mapsto e^{p-1}y$. Since $e^p=e$, this map is idempotent and a $p$-polar $k$-algebra endomorphism. Thus 
\[
A \cong \ker(f) \times \im(f)
\]
as $p$-polar algebras. Note that $\im(f)$ is a nontrivial $p$-polar algebra with unity $e$, so by the previous case, it is a product of copies of $k$. The $p$-polar algebra $\ker(f)$ has smaller dimension over $k$, so we are done by induction.
\end{proof}

\begin{lemma}\label{lemma:freeetaleextension}
The functor $\alg_k \to \Mod_\Gamma$ given by $A \mapsto \Z\langle \Hom_{\alg_k}(A,\bar k)\rangle$ factors through $\pol_p(k)$.
\end{lemma}
\begin{proof}
Let $\mathcal E$ be the full subcategory of $\pol_p(\bar k)$ of reduced $p$-polar algebras. There is a functor $\pol_p(\bar k) \to \mathcal E$ given by $A \mapsto A/\Nil(A)$. By Lemma~\ref{lemma:reducedstructure}, all objects of $\mathcal E$ are isomorphic to $\bar k^n$ for some $n \geq 0$.
Consider the set
\[
\Hom_{\pol_p(\bar k)}(\bar k^n,\bar k).
\]
One easily verifies that a linear map $\bar k^n \to \bar k$ represented by a row vector $(a_1,\dots,a_p)$ is a homomorphism of $p$-polar algebras iff it is zero or exactly one $a_i$ is nonzero, and furthermore $a_i \in \F_p$. We now construct a functor to abelian groups,
\[
\Phi\colon \mathcal E \to \Ab
\]
by defining $\Phi(\bar k^n) = \Z^n$ and for a morphism $\bar k^n \to \bar k^m$ represented by a matrix $M$, $\Phi(M)$ is the matrix $M$ with every nonzero entry replaced by $1$. Because of the special form of $M$, this is indeed a well-defined functor. Moreover, it carries the $\Gamma$-action on $A \otimes_k \bar k$ to an action by permutation matrices on $\Phi(A \otimes_k \bar k)$. We have that $\Phi(\pol_p(R)) \cong \Z\langle \Hom_{\Alg_{k}}(R,k)\rangle$ for reduced $R$ over algebraically closed $k$.

Now define
\[
\mathcal F\colon \pol_p(k) \to \Ab
\]
by $\mathcal F(A) = \Phi\Bigl((A \otimes_k \bar k)/\Nil(A \otimes_k \bar k)\Bigr)$. By the above, this is the desired extension.
\end{proof}

\begin{proof}[Proof of Thm.~\ref{thm:freeformalgroupfactorization}]
By Theorems~\ref{thm:Wittofpolar} and \ref{thm:dieudonneoffreeformal}, the functor $R \mapsto \DieuFor(\Fr(R))$ factors through $\pol_p(R)$. The Dieudonné functor $\DieuFor$ is not an equivalence on all formal groups (just on formal $p$-groups), but it is an equivalence between connected formal groups and their image. Thus it remains to show that the \'etale part $\Fr^e$ of $\Fr$ factors through $\pol_p(k)$. Using the equivalence between \'etale formal groups and $\Gamma$-modules, this case is covered by Lemma~\ref{lemma:freeetale} together with Lemma~\ref{lemma:freeetaleextension}.
\end{proof}

\bibliographystyle{alpha}
\bibliography{bibliography}

\end{document}